\RequirePackage[l2tabu, orthodox]{nag}
\documentclass[a4paper, 12pt,reqno]{amsart}

\usepackage[utf8]{inputenc}
\usepackage{microtype}

\usepackage{amsmath}
\usepackage{amsthm}
\usepackage{amsfonts}
\usepackage{amssymb}
\usepackage{mathtools}
\usepackage{tikz-cd}

\usepackage{hyperref}
\hypersetup{
    colorlinks,
    citecolor=blue,
    filecolor=green,
    linkcolor=darkgray,
    urlcolor=gray
}

\author{Boris Bilich}
\address{Department of Mathematics, HSE University, Usacheva str. 6, 119048, Moscow, Russian Federation}
\email{bilichboris1999@gmail.com}
\title[Noncommutative monoid structures on affine surfaces]{Classification of noncommutative monoid structures on normal affine surfaces}



\theoremstyle{plain}      
\newtheorem{theorem}{Theorem}[section]     
\newtheorem{corollary}[theorem]{Corollary}     
\newtheorem{lemma}[theorem]{Lemma}     
\newtheorem{proposition}[theorem]{Proposition}

\theoremstyle{remark}      
\newtheorem{example}[theorem]{Example} 
\newtheorem{remark}[theorem]{Remark} 
\newtheorem*{notation}{Notation}
\newtheorem*{acknowledgments}{Acknowledgments}
\theoremstyle{definition}
\newtheorem{definition}[theorem]{Definition} 
\newtheorem{construction}[theorem]{Construction}

\newcommand{\KK}{\mathbb{K}}
\newcommand{\ZZ}{\mathbb{Z}}
\newcommand{\GG}{\mathbb{G}}
\newcommand{\QQ}{\mathbb{Q}}

\newcommand{\tx}{\mathtt{x}}
\newcommand{\ty}{\mathtt{y}}
\newcommand{\tz}{\mathtt{z}}

\DeclareMathOperator{\Spec}{Spec}
\DeclareMathOperator{\Aut}{Aut}

\DeclarePairedDelimiter\ceil{\lceil}{\rceil}

\begin{document}
\begin{abstract}
    In 2021, Dzhunusov and Zaitseva classified two-dimensional normal affine commutative algebraic
    monoids. In this work, we extend this classification to noncommutative monoid structures on
    normal affine surfaces. We prove that two-dimensional algebraic monoids are toric. We also show
    how to find all monoid structures on a normal toric surface. Every such structure is induced by
    a comultiplication formula involving Demazure roots. We also give descriptions of opposite
    monoids, quotient monoids, and boundary divisors.
\end{abstract}
\maketitle

\section{Introduction}
An algebraic monoid is an irreducible variety $X$ together with an associative multiplication map
$\mu\colon X\times X\to X$, which is a morphism of algebraic varieties and a unit element $e\in X$.
The theory of algebraic monoids pioneered in 1980s in the works by Putcha and Renner (see
introductory monographs \cite{P88,R05}). Many contributions in the theory were made since then. The
study of algebraic monoids involved topics from algebraic geometry, representation theory, and
combinatorics. Affine algebraic monoids, i.e., the monoids, which are affine varieties, were of
particular interest.

Let $\KK$ be an algebraically closed field of characteristic zero. We assume all varieties to be
defined over $\KK$. The group of invertible elements $G(X)$ is open in the algebraic monoid~$X$.
Rittatore proved in \cite{R07} that an algebraic monoid is affine if and only if its group of
invertible elements is affine. If $G(X)$ is an algebraic torus then $X$ is called a toric monoid
and it is fully determined by the underlying variety (see \cite{N92}). The classification of
algebraic monoids with reductive $G(X)$ was a result of works of several authors (see
\cite{V95,R05}). If the group $G(X)$ is unipotent, then the monoid $X$ is the group itself. It
follows from the fact that orbits of unipotent group action on affine varieties are closed
\cite[Theorem 7.2.1]{SR17}. 

However, there were no attempts to classify general affine algebraic monoids even in low dimensions
until recently.  In \cite{ABZ20}, the authors studied commutative monoid structures on affine
surfaces.  They gave explicit formulas for possible multiplications on $\mathbb{A}^1$,
$\mathbb{A}^2$, and~$\mathbb{A}^3$. Dzhunusov and Zaitseva classified two-dimensional normal
affine commutative monoids in \cite{DZ21}. Besides toric monoids and the two-dimensional additive
group there are monoids with $G(X)\cong \GG_a\times \GG_m$. All such monoids are isomorphic to a
toric variety $X$ with a comultiplication 
\[
    \mu^*\colon \chi^u\mapsto \chi^u\otimes \chi^u(1\otimes \chi^e+\chi^e\otimes 1)^{\langle p,
    u\rangle}
\]
where $e$ is Demazure root associated to the ray generator $p$ of the cone of $X$. 

The aim of this work is to provide a full classification of two-dimensional normal affine
noncommutative monoids over $\KK$.  Together with the result of Dzhunusov and Zaitseva it finishes
the classification of monoid structures on normal affine surfaces. 

Let $X$ be a two-dimensional normal affine noncommutative algebraic monoid. Then $G(X)$ is a
two-dimensional noncommutative connected linear algebraic group. Such group is isomorphic to the
semidirect product $G_n=\GG_a\rtimes\GG_m$ for the action of $\GG_m$ on $\GG_a$ by the
multiplication on the $n$-th power for some positive integer $n$.

First of all, we prove that every two-dimensional (not necessarily normal) affine algebraic monoid
is toric (Theorem \ref{t:monistor}). It extends the result of \cite{DZ21} by dropping the
assumptions of normality and commutativity (see also \cite{AK15}). 

Then, we give a classification of two-dimensional noncommutative normal affine algebraic monoids in
terms of convex cones. The group algebra of $\KK[X]$ is a subalgebra of $\KK[G(X)]$. This
subalgebra should satisfy some conditions. Since $X$ is a toric surface, the subalgebra $\KK[X]$ is
given by some convex cone. The action of the additive subgroup $\GG_a\subset G(X)$ gives some
restrictions on the possible form of such cone (Lemma \ref{l:conetriangle}). We construct the
series of convex cones (Construction \ref{def:xynab}) satisfying those restrictions. We prove that
these cones induce non-isomorphic monoids and that every monoid is isomorphic to one of this form
(Theorem \ref{t:convconeclass}). 

We also present another approach to the classification problem. We show how to find all monoid
structures on a normal toric surface up to isomorphism. Every such structure is given by a
comultiplication
\[
    \mu^*\colon \chi^u\mapsto \chi^u\otimes \chi^u(1\otimes \chi^{e_1}+\chi^{e_2}\otimes 1)^{\langle p,
    u\rangle}
\]
where $(e_1, e_2)$ is a pair of Demazure roots associated to the ray generator $p$ (Theorem
\ref{t:demazclass}). It is a direct generalization of the comultiplication from \cite{DZ21}.

In Section \ref{s:algprop} we study algebraic properties of monoids. We describe the opposite
monoid and the quotients of monoids. Theorem \ref{t:noninv} gives a description of the subset of
non-invertible elements (the boundary divisor). We prove that it is isomorphic to $\mathbb{A}^1$ as
a variety and that there is a zero element in it (Theorem \ref{t:noninv}).

\begin{acknowledgments}
    The author is grateful to Ivan Arzhantsev for the formulation of the problem and his guidance
    during the creation of this work and to Yulia Zaitseva for useful remarks. 
\end{acknowledgments}

\section{Preliminaries}
\subsection{Notation}
Throughout the paper we denote by $\KK$ an algebraically closed field of characteristic zero. All
algebraic varieties are assumed to be defined over $\KK$ and irreducible. The additive and
multiplicative groups of $\KK$ are denoted by $\GG_a$ and $\GG_m$ respectively. We denote by $T$
the two-dimensional algebraic torus $\GG_m^2$.

\subsection{Toric surfaces}
An \emph{affine toric surface} is an affine surface $Z$ together with an open embedding of
the algebraic torus $\iota_Z\colon T\hookrightarrow Z$ such that the action of $T$ on itself extends to
an algebraic action on $Z$.  A classification of normal affine toric surfaces is well known. We
briefly recall some facts which we use later.  We refer the reader to the excellent
textbook~\cite{CLS11} for details. 

Let $M\cong \ZZ^2$ be the character lattice of $T$ and $N$ be the dual lattice of one-parametric
subgroups. For $u\in M$ we denote by $\chi^u$ the character corresponding to $u$.  We denote by
$\langle \cdot,\cdot  \rangle\colon M\times N \to \ZZ$ the natural pairing between $M$ and $N$.

Let $\sigma$ be a rational strongly convex cone in $N_\QQ = N\otimes_\ZZ \QQ$, that is, a rational
convex cone with no linear subspace in it. Then the dual cone in $M_\QQ = M\otimes_\ZZ \QQ$ is
defined as $\sigma^\vee = \{ w \in M_\QQ \colon \langle w, v \rangle \geq 0 \text{ for all } v \in
\sigma\}$.  Let $\KK[\sigma^\vee \cap M]=\oplus_{u\in\sigma^\vee \cap M} \KK \chi^u$ be the linear
span of the characters from $\sigma^\vee \cap M$ in $\KK[T]$. Then $\KK[\sigma^\vee \cap M]$ is a
finitely generated $T$-invariant subalgebra of $\KK[T]$ and $Z_\sigma = \Spec \KK[\sigma^\vee \cap
M]$ is a normal toric surface with respect to the action of $T$.

On the other hand, suppose that $Z$ is a normal affine toric surface.  Consider the set of
characters of $T$-action on $\KK[Z]$ and take its convex hull $\sigma^\vee_Z$ in $M_\QQ$. Then
$\sigma^\vee_Z$ is a dual cone of some strongly convex cone $\sigma_Z \subset N_\QQ$ and $Z$ is
canonically isomorphic to $Z_{\sigma_Z}$.

A toric surface $Z$ can be also described in terms of $M$-grading on $\KK[Z]$. A nonzero function
$f\in \KK[Z]$ is homogeneous of degree $u \in M$ if and only if $T$ acts on $f$ by
multiplication by character $\chi^u$. Then every homogeneous part of $\KK[Z]$ has dimension at most
one. 

A morphism between toric surfaces $\varphi\colon Z_1\to Z_2$ is called \emph{toric} if there
exists a homomorphism of groups $\varphi_T \colon T \to T$ such that the following diagram is
commutative:
\[
    % https://tikzcd.yichuanshen.de/#N4Igdg9gJgpgziAXAbVABwnAlgFyxMJZABgBpiBdUkANwEMAbAVxiRAA0QBfU9TXfIRQBGclVqMWbAJrdeIDNjwEiZYePrNWiEABU5fJYKKj11TVJ36u4mFADm8IqABmAJwgBbJGRA4ISKISWmwAOqH0bmgAFlgGIO5eSABM1P5IAMzmktog4ZExWAD61vKJ3oipfgGIvjh0WAxs0RAQANYg2SE64fj1RZw8rh4VWdWBaQ1NOi3tncGWeaF9dEWyNlxAA
    \begin{tikzcd}
        Z_1 \arrow[r, "\varphi"]                              & Z_2                            \\
        T \arrow[r, "\varphi_T"] \arrow[u, "\iota_{Z_1}", hook] & T \arrow[u, "\iota_{Z_2}", hook]
    \end{tikzcd}
\]
The map $\varphi_T$ induces a linear map of character lattices $\widehat{\varphi}\colon M\to M$.
Tensoring with $\QQ$ we obtain a map $\widehat{\varphi}_\QQ\colon M_\QQ \to M_\QQ$. If
$\sigma^\vee_1,\sigma^\vee_2\subset M$ are cones for $Z_1$ and $Z_2$ then
$\widehat{\varphi}_\QQ(\sigma_2)\subset \sigma_1$. Conversely, given a linear map
$\widehat{\varphi}\colon M\to M$ with $\widehat{\varphi}_\QQ(\sigma_2)\subset \sigma_1$ we have a
unique toric morphism $\varphi\colon {Z_1}\to {Z_2}$, which induces $\widehat\varphi$. 

\subsection{Demazure roots}
A derivation $\delta$ on a commutative algebra $A$ is called \emph{locally nilpotent} if for every
element $f\in A$, there exists a positive integer $n$ such that $\delta^n (f)=0$. For a toric
surface $Z$, the locally nilpotent derivation $\delta$ on $\KK[Z]$ is called \emph{homogeneous} if
it maps homogeneous elements to homogeneous. There exists a unique element $e\in M$ called the
\emph{degree} of $\delta$ such that $\delta(\chi^u)=c(u)\chi^{u+e}$, where $c(u)$ is an element of
$\KK$, which depends on $u$. 

Homogeneous locally nilpotent derivations on a toric surface have a nice combinatorial description.
Let $\sigma\subset N$ be a two-dimensional strongly convex cone. Let $p_1$ and $p_2$ be the ray
generators of $\sigma$, that is, the primitive lattice points on two edges of $\sigma$. 
\begin{definition}
   An element $e \in M$ is called a \emph{Demazure root} associated to $p_i$ if $\langle e, p_i
   \rangle = -1,$ and $\langle e, p_j \rangle \geq 0$ for $j\neq i$.
\end{definition}
A Demazure root $e$ associated to $p_i$ defines a locally nilpotent derivation $\delta_e$ by 
\[
    \delta_e(\chi^u)=\langle u, p_i \rangle\chi^{u+e}.
\]
\begin{example}
    Let $\sigma\subset N$ be a strongly convex cone with ray generators $p_1=(1,0)$ and
    $p_2=(-b, a)$ for integers $b\geq 0$ and $a>0$. The Demazure roots associated to $p_1$, are
    given by
    \[
        e_\ell = (-1, \ell) \in M
    \]
    for $a\ell + b \geq 0$.
\end{example}

\subsection{Algebraic monoids}
An \emph{affine algebraic monoid} is an affine variety $X$ together with a multiplication
morphism $\mu \colon X\times X\to X$ and a unit $\eta\colon \Spec\KK \to X$, which together satisfy
the axioms of an associative product with unit. Equivalently, the algebra $\KK[X]$ is a bialgebra
with a comultiplication $\mu^*$ and a counit $\eta^*$.

The subgroup of invertible elements of $X$ is denoted by $G(X)$. 
\begin{proposition} \label{prop:groupofmonoid}
    The group of invertible elements $G(X)$ is an affine open subset of the affine algebraic monoid
    $X$. 
\end{proposition}
\begin{proof}
    By Theorem 1 in \cite{R98}, $G(X)$ is an open subset of $X$. It is well known that every
    quasi-affine algebraic group is affine (see \cite[Theorem 8.5.4]{SR17}), so $G(X)$ is an affine open subset. 
\end{proof}

\begin{example} \label{ex:dz21monoids}
   Let $\sigma\subset N$ be a strongly convex cone. Suppose that $p\in N$ is a ray generator of
   $\sigma$ and $e\in M$ is a Demazure root, associated to $p$. Then the formula
   \[
       \chi^u\mapsto \chi^u\otimes \chi^u (1\otimes \chi^{e}+\chi^{e}\otimes 1)^{\langle p_i, u
       \rangle}
   \]
   defines a comultiplication on $\KK[\sigma^\vee \cap M]$. We call it the Dzhunusov-Zaitseva
   comultiplication. This comultiplication turns $Z_\sigma$ into a monoid. These monoids were
   studied in \cite{DZ21}.
\end{example}

\subsection{Two-dimensional solvable groups}
In order to classify noncommutative two-dimensional monoids we should start with two-dimensional
groups. Fortunately, all such groups are well-known and can be described easily.

\begin{notation}
   For every positive integer $n$ we denote by $G_n$ the semidirect product $\GG_a \rtimes \GG_m$, where
   $\GG_m$ acts on $\GG_a$ by multiplication on $n$-th power. Explicitly, the product in $G_n$ is
   given by
   \[
       (\alpha_1, \tau_1)\cdot(\alpha_2, \tau_2) = (\alpha_1 + \tau_1^n \alpha_2, \tau_1\tau_2).
   \]
\end{notation}
The group algebra of $G_n$ is $\KK[x, y, y^{-1}]$, where $x$ is a coordinate on $\GG_a$ and $y$ is
a coordinate on $\GG_m$. The comultiplication is given by 
\[
    \begin{split}
        x &\mapsto x\otimes 1 + y^n\otimes x, \\
        y &\mapsto y\otimes y.
    \end{split}
\]
On an arbitrary monomial $x^ay^b$ with integer $b$ and non-negative integer $a$ we have
\begin{equation} \label{eq:comultipongn}
    x^ay^b\mapsto (x\otimes 1 + y^n\otimes x)^a (y\otimes y)^b.
\end{equation}

\begin{proposition} \label{prop:clasof2groups}
   Let $G$ be a two-dimensional noncommutative connected linear algebraic group. Then $G$ is
   isomorphic to $G_n$ for some positive integer $n$.
\end{proposition}
\begin{proof}
    The group $G$ is isomorphic to a semidirect product of reductive and solvable subgroups by
    {Levi-Malcev} decomposition. By the classification of reductive groups we deduce that $G$ is
    solvable.

    If $G$ is unipotent, then it has a center of dimension at least one. But $G$ has dimension two,
    so it is commutative. That means that $G$ is not unipotent.

    By the structure theorem for solvable groups, $G$ is a semidirect product of its maximal torus
    and the unipotent radical. Both must have dimension $1$ so it is a semidirect product of
    $\GG_m$ and $\GG_a$, which is isomorphic to $G_n$ for some $n$. For the details about the
    structure of solvable groups see \cite[Chapter VII]{H75}.
\end{proof}
\begin{proposition} \label{prop:centreofgn}
    The center $\mathcal{Z}(G_n)$ of the group $G_n$ is a cyclic group of order $n$. It is generated by an element
    $(0, \xi_n)$, where $\xi_n$ is a primitive $n$-th root of unity. Moreover, if $C_m \subset
    \mathcal{Z}(G_n)$ is a subgroup of order $m$, then
    \[
        G_n/C_m\cong G_{k},
    \]
    where $k=n/m$.
\end{proposition}
\begin{proof}
    Suppose that $(\alpha, \tau)$ lies in the center of $G_n$. For an arbitrary $(\alpha',
    \tau')\in G_n$, we have $(\alpha, \tau)\cdot(\alpha', \tau')=(\alpha', \tau')\cdot (\alpha,
    \tau)$, so
    \[
        (\alpha + \tau^n \alpha', \tau \tau')=(\alpha' + {\tau'}^n\alpha, \tau \tau').
    \]
    Thus, we have $\alpha=0$ and $\tau^n=1$. It follows that the center consists of elements of the
    form $(0,\tau)$ with $\tau^n=1$ so it is generated by $(0,\xi_n)$.

    The subgroup $C_m\subset \mathcal{Z}(G_n)$ is generated by the element $(0,\xi_m)$ where
    $\xi_m=\xi_n^{k}$. Define a surjective homomorphism $\pi\colon G_n\to G_{k}$ by
    $\pi(\alpha, \tau)=(\alpha, \tau^m)$. Let us check that $\pi$ is a group homomorphism. For
    $(\alpha, \tau),(\alpha', \tau')\in G_n$, we have
    \[
        \pi((\alpha, \tau)\cdot(\alpha', \tau'))=\pi(\alpha + \tau^n \alpha', \tau \tau')=(\alpha +
        \tau^n \alpha', \tau^m (\tau')^m)=\pi(\alpha, \tau)\cdot\pi(\alpha', \tau').
    \]
    The kernel of $\pi$ consists of the elements of the form $(0,\tau)$ with $\tau^m=1$, so
    $\ker\pi=C_m$. Thus, $\pi$ descends to the isomorphism $G_n/C_m\cong G_{k}$.
\end{proof}
\begin{corollary}
    The groups $G_{n}$ and $G_{n'}$ are isomorphic if and only if $n=n'$. 
\end{corollary}

\section{Noncommutative monoid structures on normal affine surfaces}

The aim of this work is to classify and to study noncommutative monoid structures on normal affine
algebraic surfaces. Combining Propositions \ref{prop:groupofmonoid} and
\ref{prop:clasof2groups}, we deduce that every such surface admits an embedding of the group $G_n$
as an open subset such that the actions of $G_n$ on itself both by right and left multiplication
extends to the action on the whole surface. In the next paragraph we classify all such embeddings.

\subsection{Classification in terms of convex cones}
First of all, let us learn more about left and right actions of $G_n$ on itself.  Define a map
$\rho_n\colon G_n\times G_n \to \Aut(G_n)$ as 
\[
\rho_n(g_1,g_2)(g)= g_1gg_2^{-1}
\]
for $g_1,g_2,g\in G_n$.
The kernel of this map is the diagonal embedding of the center $\Delta(\mathcal{Z}(G_n))\subset
G_n\times G_n$. Thus, the image is a four-dimensional linear algebraic group with a maximal torus
of dimension $2$. 
\begin{proposition} \label{prop:torongn}
   There is an embedding of the two-dimensional torus $\nu\colon T=\GG_m^2\hookrightarrow
   \rho_n(G_n\times G_n)$ such that
   \[
       \nu((t_1, t_2))\cdot (\alpha, \tau) = (t_1^{-1}\alpha, t_2^{-1}\tau)
   \]
   for $(t_1,t_2)\in T$ and $(\alpha, \tau)\in G_n$. This action endows $G_n$ with the structure of
   a toric surface.
\end{proposition}
\begin{proof}
    The action of an element $\rho_n((0, \tau_1)\times(0, \tau_2))$ on $(\alpha,\tau)$ is given by 
    \[
        (0, \tau_1)(\alpha,\tau)(0, \tau_2^{-1})=(\tau_1^n\alpha, \tau_1\tau_2^{-1}\tau).
    \]
    Thus, we can take $\nu((t_1,t_2)) = \rho((0, {t_1}^{-\frac{1}{n}})\times (0,
    {t_1}^{-\frac{1}{n}}t_2))$. The kernel of $\rho$ consists of elements of the form
    $(0,\xi)\times(0,\xi)$ with $\xi^n=1$, so the choice of $n$-th root does not affect the
    definition. The action is effective, so it turns $G_n$ into a normal toric surface.
\end{proof}
Speaking of $G_n$ as a toric surface we always assume the toric structure from Proposition~\ref{prop:torongn} above. The grading on $\KK[G_n]=\KK[x, y, y^{-1}]$ is given by $\deg x = (1,
0),~\deg y = (0, 1)$.  Cones for toric surfaces $G_n$ are given by
\begin{equation}\label{eq:gncones}
\begin{split}
    \sigma_{G_n}=&\{(a, 0)\in N_\QQ\colon a\geq 0\}, \\
    \sigma_{G_n}^\vee=&\{(a, b)\in N_\QQ \colon a\geq 0,~b\in \QQ \}.
\end{split}
\end{equation}

Now we are ready for one of the main results.
\begin{theorem} \label{t:monistor}
    Every two-dimensional (not necessarily normal) affine algebraic monoid admits a toric structure.
\end{theorem}
\begin{proof}
    The case of normal commutative monoids was proved in \cite[Theorem 4]{DZ21}. However, the
    statement is also true without this assumption. Suppose that $X$ is a commutative
    two-dimensional affine algebraic monoid. If $G(X)$ is a two-dimensional torus, then $X$ is
    toric by definition. If the maximal torus of $G(X)$ is trivial then $G(X)$ is unipotent. By
    \cite[~Theorem 7.2.1]{SR17} the orbits of $G(X)$ acting on $X$ are closed. Since $G(X)$ is open
    in $X$, we have $X=G(X)\cong \GG_a^2$, which is toric.

    Assume now that $G(X)$ has the maximal torus of dimension one, so $G(X)\cong \GG_a\times\GG_m$.
    The algebra of functions on $G(X)$ is isomorphic to $\KK[x,y,y^{-1}]$ with the induced action
    of $G(X)$ given by 
    \[
        (\alpha, \tau)\cdot x^ay^b=\tau^{-b}(x-\alpha)^ay^b
    \]
    for $(\alpha,\tau)\in \GG_a\times \GG_m$ and integers $a\geq0,~b$.

    The algebra of functions $\KK[X]$ on $X$ is a subalgebra of $\KK[G(X)]$. The action of $G(X)$
    on itself extends to the whole $X$, so $\KK[X]$ is invariant under the action of $G(X)$. The
    action of $\GG_m \subset G(X)$ induces a grading on $\KK[G(X)]$ with $\deg(x^ay^b)=-b$ and
    $\KK[X]$ is a graded subalgebra.

    Take an arbitrary homogeneous element $f\in \KK[X]$. It has the form $f=y^bg(x)$, where $g(x)$
    is a polynomial in $x$. The action of $\GG_a\subset G(X)$ on $f$ is given by $(\alpha, 0)\cdot
    f=y^bg(x-\alpha)$. The polynomials of the form $g(x-\alpha)$ for $\alpha\in \KK$ span the
    vector subspace of polynomials of degree not greater than the degree of $g$. This means that
    $x^ay^b\in \KK[X]$ for $a\leq \deg g$. In other words, all the monomials in $f$ are in
    $\KK[X]$. It follows that $\KK[X]$ is the linear span of monomials $x^ay^b\in \KK[X]$ and the
    grading $\deg(x^ay^b)=(-a,-b)$ defines a toric structure on~$X$. This proves the commutative
    case.

    Now we may assume that $X$ is a noncommutative two-dimensional affine algebraic monoid. The
    group of invertible elements $G(X)$ is a two-dimensional noncommutative connected linear
    algebraic group, so it is isomorphic to $G_n$ for some $n$ by Proposition
    \ref{prop:clasof2groups}. 

    Since $\rho(G_n\times G_n)$ is given by right and left multiplications, every automorphism from
    $\rho(G_n\times G_n)$ extends uniquely to an automorphism of $X$. The toric structure on $G_n$
    from Proposition \ref{prop:torongn} is given by an appropriate embedding $T\hookrightarrow
    \rho(G_n\times G_n)$, so the action of $T$ also extends to $X$.
\end{proof}
\begin{corollary} \label{cor:monistor}
    Let $X$ be an algebraic monoid with the group of invertible elements isomorphic to $G_n$. There
    exists a toric structure on $X$ such that the inclusion $G_n\hookrightarrow X$ is toric.
\end{corollary}
\begin{proof}
    It follows immediately from the proof of Theorem \ref{t:monistor}.
\end{proof}

From now on we fix $n$ and assume that $X$ is a normal affine monoid with the group of
invertibles $G_n$. Thus, $X$ is given by the strongly convex cone $\sigma_X \supset \sigma_{G_n}$
in $N$ or equivalently by $\sigma_X^\vee \subset \sigma_{G_n}^\vee$ in $M$. We are more
interested in the cone $\sigma_X^\vee\subset M$ and we obtain a classification of monoids from
the classification of such cones with additional conditions. 

So far, we used only left and right $\GG_m$ actions on $X$.  Action of $\GG_a\subset G_n$ on $G_n$
by left and right multiplications induces two locally nilpotent derivations on
$\KK[G_n]=\KK[x,y,y^{-1}]$. Let us write down the
action of the element $(\alpha, 0)\in G_n$ on the element $(\alpha_1, \tau_1)$:
\[
    (\alpha, 1)\cdot (\alpha_1, \tau_1) = (\alpha + \alpha_1, \tau_1),
\]
\[
    (\alpha_1, \tau_1)\cdot (\alpha_, 1) = (\alpha_1 + \tau^n \alpha, \tau_1).
\]
Hence, the induced derivations are proportional to $\partial_x$ and $y^n\partial_x$ respectively.
\begin{notation}
    We denote by $\delta_l$ and $\delta_r$ locally nilpotent derivations on $\KK[G_n]$ given
    by $\delta_l = \partial_x$ and $\delta_r=y^n\partial_x$. These locally nilpotent
    derivations are induced by left and right $\GG_a$ actions on $G_n$.
\end{notation}
\begin{remark} \label{rem:lndconstant}
    Note that these derivations are defined up to scalar. This scalar depends on the choice of
    parametrization of $\GG_a$ in $G_n$. 
\end{remark}
The actions of $\GG_a$ extends to the whole $X$, so $\KK[X]$ should be invariant for $\delta_l$ and
$\delta_r$. We use this fact for the next lemma. Recall that the group algebra of $G_n$ is
identified with the semigroup algebra $\KK[\sigma^\vee_{G_n} \cap M]$, where $\sigma^\vee_{G_n}$ is
defined in \eqref{eq:gncones}. Every subcone $\sigma^\vee\subset \sigma^\vee_{G_n}$ defines a graded
subalgebra $\KK[\sigma^\vee \cap M]\subset \KK[G_n]$.
\begin{lemma} \label{l:conetriangle}
    Let $\sigma^\vee\subset \sigma^\vee_{G_n}$ be a convex cone. Then $\KK[\sigma^\vee \cap M]$
    is a subalgebra in $\KK[G_n]$ and the following are equivalent.
    \begin{enumerate}
        \item The comultiplication on $\KK[G_n]$ restricts to a comultiplication on $\KK[\sigma^\vee
            \cap M]$. 
        \item For every lattice point $(a,b)$ in $\sigma^\vee$ the points $(0,b)$ and $(0,b+na)$
            are also in $\sigma^\vee$.
    \end{enumerate}
\end{lemma}
\begin{proof}
    The point $(a,b)$ is in the cone $\sigma^\vee$ if and only if $x^ay^b\in \KK[\sigma^\vee \cap
    M]$. 

    First, suppose that (1) holds. Then the derivations $\delta_l$ and $\delta_r$ also restrict
    to $\KK[\sigma^\vee \cap M]$ and we can write
    \[
        \delta_l^{a} x^ay^b = a!y^b,
    \]
    \[
        \delta_r^{a} x^ay^b = a!y^{b+na}.
    \]
    That means that $(0,b)$ and $(0,b+na)$ are in $\sigma^\vee$.

    Now suppose that (2) holds. Let us expand the brackets in the formula for comultiplication
    \eqref{eq:comultipongn}:
    \[
        x^ay^b\mapsto (x\otimes 1 + y^n\otimes x)^a (y\otimes y)^b = \sum_{i=0}^{a}
        \binom{a}{i} x^{a-i}y^{b+ni}\otimes x^iy^b.
    \]
    Since $\sigma^\vee$ is convex and it contains points $(a,b)$, $(0,b)$, and $(0, b+na)$, it also
    contains the points $(a-i, b+ni)$ and $(i,b)$ for $0\leq i \leq a$. It follows that the
    comultiplication restricts to $\KK[\sigma^\vee\cap M]$.
\end{proof}
 We are going to introduce a notation for subcones of $\sigma^\vee_{G_n}$, which satisfy the second
 condition of Lemma $\ref{l:conetriangle}$. We will eventually find out that the corresponding toric
 surfaces are pairwise non-isomorphic noncommutative monoids.
\begin{construction} \label{def:xynab}
    Fix a positive integer $n$. Suppose that $a>0$ and $b\geq 0$ are integers and $\gcd(a,b)=1$. We
    define series of subcones of $\sigma^\vee_{G_n}$.
    \begin{enumerate}
        \item We denote by $\sigma^{\vee}(n, a, b)$ the subcone of $\sigma^\vee_{G_n}$, generated
            by rays $(0, 1)$ and $(a, b)$. Then, $\KK[\sigma^{\vee}(n, a, b) \cap M]$ is a subalgebra
            of $\KK[G_n]$ and we denote the corresponding surface as $X_n^{a,b} = \Spec
            \KK[\sigma^{\vee}(n, a, b) \cap M]$.
        \item We denote by $\bar\sigma^{\vee}(n, a, b)$ the subcone of $\sigma^\vee_{G_n}$,
            generated by rays $(0, -1)$ and ${(a, -na - b)}$. Then, $\KK[\bar\sigma^{\vee}(n, a, b)
            \cap M]$ is a subalgebra of $\KK[G_n]$ and we denote the corresponding surface as
            $Y_n^{a,b} = \Spec \KK[\bar\sigma^{\vee}(n, a, b) \cap M]$.
    \end{enumerate}
\end{construction}
Note that, by definition, $X_n^{a, b}$ and $Y_n^{a,b}$ come with a natural open embedding of $G_n$.
\begin{proposition}
    The multiplication on $G_n$ extends to a multiplication on $X_n^{a,b}$ or $Y_n^{a, b}$. Hence
    $X_n^{a, b}$ and $Y_n^{a, b}$ are monoids.
\end{proposition}
\begin{proof}
    It is enough to prove that the cones $\sigma^\vee(n,a,b)$ and $\bar\sigma^\vee(n,a,b)$ satisfy
    the second condition of Lemma \ref{l:conetriangle}. 

    First, let us consider the case of $\sigma^\vee(n,a,b)$. Every point $(a',b')$ in
    $\sigma^\vee(n,a,b)$ has $b'>0$. Moreover, every point of the form $(0,b'')$ with $b''\geq0$ lies
    in $\sigma^\vee(n,a,b)$. Hence, if $(a',b')$ in $\sigma^\vee(n,a,b)$, then $(0,b')$ and
    $(0,b'+na')$ are in $\sigma^\vee(n,a,b)$.

    The case of $\bar\sigma^\vee(n,a,b)$ is a bit more complicated. Again all the points of the
    form $(0,b'')$ with $b''\leq 0$ lie in $\bar\sigma^\vee(n,a,b)$. For every point $(a',b')$ in
    $\bar\sigma^\vee(n,a,b)$ we have $b'\leq -na'$, hence both $b'$ and $b'+na'$ are non-positive.
    Thus, the cone $\bar\sigma^\vee(n,a,b)$ also satisfies the second condition of Lemma
    \ref{l:conetriangle}.
\end{proof}
We now want to prove that $X_n^{a,b}$ and $Y_n^{a,b}$ are pairwise non-isomorphic. To do this let
us introduce a sequence of invariants of monoids associated to image ideals of $\delta_l$. 

Let $X$ be an algebraic monoid with the group of invertible elements $G_n$. The locally nilpotent
derivation $\delta_l$ on $\KK[G_n]$ restricts to a locally nilpotent derivation on $\KK[X]$. Then
the subspace $\ker(\delta_l)\cap\delta_l^k(\KK[X])$ is an ideal in
$\ker(\delta_l)\cap\KK[X]$. It is called the $k$-th image ideal of $\delta_l$ (see
\cite[Proposition 1.9]{F17}).
\begin{notation}
    We denote by $\mathfrak{L}_k(X)$ the codimension of $k$-th image ideal of $\delta_l$:
    \[
        \mathfrak{L}_k(X)=\dim \frac{\ker(\delta_l)\cap
        \KK[X]}{\ker(\delta_l)\cap\delta_l^k(\KK[X])}.
    \]
\end{notation}
Note that, by Remark \ref{rem:lndconstant}, the locally nilpotent derivation $\delta_l$ is
defined up to scalar. However, the definition of $\mathfrak{L}_k(X)$ does
not depend on this constant.
\begin{lemma} \label{l:linvariant}
    For monoids $X_n^{a,b}$ and $Y_n^{a,b}$ we have 
    \[
        \mathfrak{L}_k(X_n^{a,b})= \ceil*{\frac{k}{a}b},
    \]
    \[
        \mathfrak{L}_k(Y_n^{a,b})= \ceil*{\frac{k}{a}(b+na)}=\ceil*{\frac{k}{a}b}+nk
    \]
    where $\ceil*{\cdot}$ denotes the ceiling function.
\end{lemma}
\begin{proof}
    In case of $X=X_n^{a,b}$ the kernel $\ker(\delta_l)\cap \KK[X]$ is equal to $\KK[y]$. Thus, the
    ideal $\ker(\delta_l)\cap \delta_l^k(\KK[X])$ is the linear span of the elements of the form
    $\delta_l^k(x^ky^{a'})=k!y^{a'}$. Hence, 
    \[
        \mathfrak{L}_k(X)=\min \{a'\geq 0 \colon (k, a')\in \sigma^\vee(n,a,b)\}=\ceil*{\frac{k}{a}b}.
    \]
    The case of $X=Y_n^{a,b}$ is treated in a similar way.
\end{proof}
We can now use invariants $\mathfrak{L}_k$ to distinguish between different non-isomorphic
monoids.
\begin{proposition}
    The monoids from Construction \ref{def:xynab} are pairwise non-isomorphic.
\end{proposition}
\begin{proof}
   Two monoids are not isomorphic if their groups of invertible elements are not isomorphic, so we
   may fix $n$.

   Consider monoids $X_n^{a,b}$ and $X_n^{a',b'}$ for $(a,b)\neq(a',b')$ with integers $a,a'>0$,
   $b,b'\geq 0$, and $\gcd(a,b)=\gcd(a',b')=1$.  By Lemma \ref{l:linvariant} we have
   $\mathfrak{L}_{aa'}(X_n^{a,b})=a'b$ and $\mathfrak{L}_{aa'}(X_n^{a',b'})=ab'$. Since
   $\gcd(a,b)=\gcd(a',b')$, we have $a'b\neq ab'$, so the monoids $X_n^{a,b}$ and $X_n^{a',b'}$ are
   not isomorphic. The same argument works for $Y_n^{a,b}$.

   It remains to show that $X=X_n^{a,b}$ and $Y=Y_n^{a',b'}$ are non-isomorphic. If $X$ and $Y$
   were isomorphic then there would be an isomorphism between the algebras $\KK[X]$ and $\KK[Y]$
   such that the derivations $\delta_l$, $\delta_r$ on $\KK[X]$ map to the corresponding
   derivations on $\KK[Y]$. The element $x^n$ of $\KK[X]$ has the property that
   $\delta_r=x^n\delta_l$ but there is no such element in $\KK[Y]$.  So there is no isomorphism
   between $\KK[X]$ and $\KK[Y]$, which preserves $\delta_l$ and $\delta_r$.
\end{proof}

We are now ready for the main classification result.
\begin{theorem} \label{t:convconeclass}
   Let $X$ be a normal affine two-dimensional noncommutative algebraic monoid. Then $X$ is
   isomorphic to one of the following.
   \begin{enumerate}
       \item The group $G_n$ for the unique $n$. 
       \item The monoid $X_n^{a,b}$ for the unique integers $n>0$, $a>0$, and $b\geq0$ with
           $\gcd(a,b)=1$.
       \item The monoid $Y_n^{a,b}$ for the unique integers $n>0$, $a>0$, and $b\geq0$ with
           $\gcd(a,b)=1$.
   \end{enumerate}
\end{theorem}
\begin{proof}
    If $X$ is a group, then, by Proposition \ref{prop:clasof2groups}, it is isomorphic to $G_n$ for
    the unique $n$.
    
    So, we assume that $X$ is not a group and $G(X)\cong G_n$ for some $n$. Fix an embedding
    of the group of invertible elements $G_n \hookrightarrow X$. Then, by Corollary \ref{cor:monistor},
    the toric structure on $G_n$ extends to a toric structure on $X$ and $\KK[X]$ is obtained from a
    convex two-dimensional proper subcone $\sigma^\vee_X\subset \sigma^\vee_{G_n}$. Moreover, the
    subcone $\sigma^\vee_X$ satisfies the second condition of Lemma $\ref{l:conetriangle}$. 

    Particularly, $\sigma^\vee_X$ has points on $y$-axis, so either $(0,1)$, or $(0,-1)$ are in
    $\sigma^\vee_X$. Since $\sigma^\vee_X$ is a proper subcone of $\sigma^\vee_X$, only one of
    these points is in $\sigma^\vee_X$.

    Suppose that $(0,1)\in \sigma^\vee_X$. If $(a',b')\in\sigma^\vee_X$, then $(0,b')$ is in
    $\sigma^\vee_X$, so $b'\geq 0$. So, there are unique $a$ and $b$ such that
    $\sigma^\vee_X=\sigma^\vee(n,a,b)$ and $X$ is isomorphic to $X_n^{a,b}$.

    Suppose now that $(0,-1)\in \sigma^\vee_X$. If $(a',b')\in\sigma^\vee_X$, then $(0,b'+na')$ is
    in $\sigma^\vee_X$, so $b'\leq -na'$. Again, there are unique $a$ and $b$ such that
    $\sigma^\vee_X=\bar\sigma^\vee(n,a,b)$.
\end{proof}
\begin{corollary} \label{cor:affmonoids} 
    Every noncommutative monoid structure on the affine space $\mathbb{A}^2$ up to isomorphism is
    given on elements $(\tx_1,\ty_1),(\tx_2,\ty_2)\in \mathbb{A}^2$ either by multiplication
    \[
        (\tx_1, \ty_1)\cdot (\tx_2, \ty_2) = (\tx_1\ty_2^b+\ty_1^{b+n}\tx_2,
        \ty_1\ty_2),
    \]
    or by
    \[
        (\tx_1, \ty_1)\cdot (\tx_2, \ty_2) = (\tx_1\ty_2^{b+n}+\ty_1^{b}\tx_2,
        \ty_1\ty_2)
    \]
    for some integers $n>0$ and $b\geq 0$.
\end{corollary}
\begin{proof}
    The toric surface is isomorphic to $\mathbb{A}^2$ if and only if ray generators of the
    corresponding cone $\sigma^\vee_{\mathbb{A}^2}$ form a basis of $M$. Thus, all the
    noncommutative monoid structures on $\mathbb{A}^2$ are isomorphic to $X_n^{1,b}$ or
    $Y_n^{1,b}$.

    The group algebra $\KK[X_n^{1,b}]$ is generated freely by the elements $y$ and $xy^b$. Writing
    down the comultiplication we obtain
    \[
        xy^b\mapsto xy^b\otimes y^b  + y^{b+n}\otimes xy^b,
    \]
    \[
        y \mapsto y\otimes y.
    \]
    This comultiplication induces the first multiplication from the statement. 

    In case of $Y_n^{1,b}$ we use the generators $y^{-1}$ and $xy^{-b-n}$ with comultiplication
    \[
        xy^{-b-n}\mapsto xy^{-b-n}\otimes y^{-b-n}  + y^{-b}\otimes xy^{-b-n},
    \]
    \[
        y^{-1} \mapsto y^{-1}\otimes y^{-1}.
    \]
    This comultiplication induces the second multiplication.
\end{proof}

\begin{example}
    Consider the monoid $X_n^{2,2k+1}$ for integers $n>0,~k\geq 0$. The algebra $\KK[X_n^{2,2k+1}]$
    has generators \[f_1=y,~f_2=xy^{k+1},\text{ and } f_3=x^2y^{2k+1}\] with the only relation
    $f_2^2=f_1f_3$ (see \cite[Example 1.1.6]{CLS11}). So $X_n^{2,2k+1}$ is isomorphic to a closed
    subvariety of $\mathbb{A}^3$ given by $\{(\tx, \ty, \tz)\in \mathbb{A}^3\colon \tx\tz=\ty^2\}$.

    The comultiplication is given by 
    \begin{align*}
        f_1&=y\mapsto y\otimes y=f_1\otimes f_1,\\ 
        f_2&=xy^{k+1}\mapsto xy^{k+1}\otimes y^{k+1}+y^{n+k+1}\otimes xy^{k+1}=f_2\otimes
        f_1^{k+1}+f_1^{n+k+1}\otimes f_2, \\
        f_3&=x^2y^{2k+1}\mapsto x^2y^{2k+1}\otimes y^{2k+1}+2(xy^{n+2k+1}\otimes xy^{2k+1})+
        y^{n+2k+1}\otimes x^2y^{2k+1}=\\ & =f_3\otimes f_1^{2k+1}+2(f_2f_1^{n+k}\otimes
        f_2f_1^{k})+f_1^{n+2k+1}\otimes f_3.
    \end{align*}
    Thus, the multiplication of the elements $(\tx_1, \ty_1, \tz_1), (\tx_2, \ty_2,\tz_2)\in
    X_n^{2,2k+1}$ is given by
    \[
        (\tx_1, \ty_1, \tz_1) \cdot (\tx_2, \ty_2,\tz_2)=(\tx_1 \tx_2,
        \ty_1\tx_2^{k+1}+\tx_1^{n+k+1}\ty_2,
        \tz_1\tx_2^{2k+1}+2\tx_1^{n+k}\ty_1\tx_2^{k}\ty_2+\tx_1^{n+2k+1}\tz_2).
    \]

    Observe that here we can take $n=0$ and obtain the commutative comultiplication from
    \cite[Example 10]{DZ21}.
\end{example}

\subsection{Demazure roots and comultiplications}

Now we are going to describe the monoid structures on the toric surface in terms of Demaure roots.
We say that an algebraic monoid $X$ has rank 1, if the maximal torus of $G(X)$ is one-dimensional.
We have already seen that all noncommutative two-dimensional monoids have rank 1. This paragraph
covers also commutative two-dimensional monoids of rank 1 as a special case.

Let $\sigma \subset N$ be a strongly convex two-dimensional cone. Suppose that $p \in N$ is a
ray generator of $\sigma$ and $e\in M$ is a Demazure root associated to $p$. Let $v$ be a ray
generator of $\sigma^\vee$ dual to $p$, that is, $\langle p, v \rangle = 0$. We need the following
lemma.
\begin{lemma} \label{l:demazbasis}
    The pair $(-e, v)$ form a basis of $M$.
\end{lemma}
\begin{proof}
    Take arbitrary $w\in M$. Observe that $w+\langle p, w\rangle e$ is orthogonal to $p$. Since
    $v$ is primitive, $w+\langle p, w\rangle e$ equals $bv$ for some integer $b$ and $w=bv-\langle
    p, w\rangle e$. It follows that every vector can be expressed as a sum of $-e$ and $v$ with
    integer coefficients, hence $(-e,v)$ is a basis of $M$.
\end{proof}
In the basis from Lemma \ref{l:demazbasis} we have $\delta_e=\partial_{x}$ for $x=\chi^{(1,0)}$. This basis
plays a crucial role in the proof of the next theorem.

\begin{theorem}\label{t:demazclass}
Suppose that $\sigma \subset N$ is a strongly convex two-dimensional cone with ray generators
$p_1$ and $p_2$. Let $(e_1, e_2)$ be a pair of Demazure roots associated to $p_i$ for some fixed
$i$. The map 
\begin{equation} \label{eq:demazcomult}
    \chi^u\mapsto \chi^u\otimes \chi^u (1\otimes \chi^{e_1}+\chi^{e_2}\otimes 1)^{\langle p_i, u
    \rangle}
\end{equation}
defines a comultiplication on $\KK[Z_\sigma]$  and the
corresponding monoid structure on $Z_\sigma$ has rank~1. Every monoid structure of rank 1 on
$Z_\sigma$ is isomorphic to such structure for some pair $(e_1, e_2)$.
\end{theorem}
\begin{proof}
    In case $e_1=e_2$ we obtain a cocommutative comultiplication from \cite[Theorem 3]{DZ21}. It is proved there
    that every commutative monoid structure is isomorphic to one of this form (see also Example~\ref{ex:dz21monoids}). So, we may assume that $e_1\neq e_2$.

    Formula \eqref{eq:demazcomult} does not depend on the basis. So we may use Lemma
    \ref{l:demazbasis} and take the basis $(-e_1, v_i)$, where $v_i$ is the ray generator of
    $\sigma^\vee$, dual to $p_i$. In this basis, the second ray generator of $\sigma^\vee$ has the
    form $(a,b)$ for $a> 0$. The ray generator $p_j$ for $j\neq i$ is given by $(-b, a)$ in the
    dual basis. The second Demazure root has the form $e_2=(-1, n)$ for some nonzero integer $n$.
    The nonnegativity condition in the definition of Demazure roots implies two inequalities:
    \[
        \langle p_j, e_1\rangle = b \geq 0,
    \]
    and 
    \[
        \langle p_j, e_2 \rangle = b+na \geq 0.
    \]

    We now have two possible cases depending on the sign of $n$. 

    If $n>0$, write $x$ for $\chi^{(1,0)}$ and $y$ for $\chi^{(0,1)}$. Then formula
    \eqref{eq:demazcomult} transforms into
    \[
        x^{a'}y^{b'}\mapsto x^{a'}y^{b'}\otimes x^{a'}y^{b'} (1\otimes x^{-1} +
        x^{-1}y^n \otimes 1)^{a'}=(y\otimes y)^{b'}(x\otimes 1 + y^n\otimes x)^{a'},
    \]
    which is the same as the comultiplication in \eqref{eq:comultipongn}. Since the cone
    $\sigma^\vee$ has ray generators $(a,b)$ and $(0,1)$ with $a>0$ and $b\geq 0$, this is a
    comultiplication, which turns $Z_\sigma$ into a monoid, which is isomorphic to $X^{a,b}_n$.

    Suppose now that $n<0$. This time write $x$ for $\chi^{(1,0)}$ and $y$ for $\chi^{(0,-1)}$.
    Using the same argument, we obtain the comultiplication in \eqref{eq:comultipongn} and the
    corresponding monoid structure is isomorphic to $Y_{-n}^{a, b+na}$.

\end{proof}

Let $\widehat\tau\colon M_\QQ \to M_\QQ$ be the linear map, which swaps the ray generators of
$\sigma^\vee$. 
\begin{corollary}
   If $\widehat\tau(M)\neq M$, then all the monoid structures on $Z_\sigma$ from Theorem
   \ref{t:demazclass} are pairwise non-isomorphic. If $\widehat\tau(M) = M$, then the only isomorphisms
   are between the monoid structures induced by pairs $(e_1, e_2)$ and $(\widehat\tau(e_1),
   \widehat\tau(e_2))$.
\end{corollary}
\begin{proof}
    Again, the commutative case is proven in \cite[Corollary 1]{DZ21}, so we consider pairs of
    non-equal Demazure roots.  We have seen in the proof of Theorem \ref{t:demazclass} that a
    noncommutative monoid defined by a pair of Demazure roots is isomorphic to $X_n^{a,b}$ or
    $Y_n^{a,b}$. What is more important that the isomorphism from the proof is toric. This implies
    that if two such monoids are isomorphic then the isomorphism can be chosen to be toric.

    If $(e_1, e_2)$ and $(e_1', e_2')$ are non-equal pairs of Demazure roots associated to a
    single ray generator $p_i$, then they are not isomorphic, since there is no toric automorphism,
    which maps $(e_1, e_2)$ to $(e_1', e_2')$. 

    If $(e_1, e_2)$ and $(e_1', e_2')$ are associated to different ray generators, then the only
    possible toric automorphism, which maps $(e_1, e_2)$ to $(e_1', e_2')$ should be induced by
    $\widehat\tau$. But $\widehat\tau$ induces an automorphism if and only if $\widehat\tau(M)=M$.
\end{proof}
\begin{remark}
    Note that there is an inaccuracy in the second part of \cite[Corollary 1]{DZ21}. It states that there are infinitely many non-isomorphic commutative monoid
    structures on the surface with one-dimensional cone. The variety corresponding to
    one-dimensional cone is $\mathbb{A}^1\times\KK^\times$. Every commutative monoid structure of rank one on
    this variety should be isomorphic to the group $\GG_a\times\GG_m$. That is why we assume that the cone is full-dimensional in Theorem \ref{t:demazclass}.
\end{remark}

\section{Algebraic properties of monoids} \label{s:algprop}

\subsection{Opposite monoids}

Let $X$ be an algebraic monoid. By $X^{op}$ we denote the \emph{opposite} algebraic monoid. It is
isomorphic to $X$ as an algebraic variety and the multiplication is given by reversing the order:
\[ \tx^{op} \ty^{op} = (\ty \tx)^{op}, \] where $\tx, \ty\in X$ and $-^{op}$ is the canonical
variety isomorphism from $X$ to $X^{op}$. Note that $X$ is canonically isomorphic to
$(X^{op})^{op}$ with respect to the map $(-^{op})^{op}$. Every group is isomorphic to its opposite
via the inversion map $\tx \mapsto (\tx^{-1})^{op}$.

Let $\mu_X^*\colon \KK[X]\to \KK[X]\otimes_\KK \KK[X]$ be a comultiplication for $X$. We can
identify the algebra $\KK[X^{op}]$ with $\KK[X]$. Then the comultiplication for $X^{op}$ is given
by $\mu_{X^{op}}^*=\varkappa \circ \mu_X$, where $\varkappa(f\otimes g)=g\otimes f$ for
$f,g\in\KK[X]$.

\begin{proposition}
    The monoid $Y_n^{a,b}$ is isomorphic to the opposite of $X_n^{a,b}$:
   \[
       Y_n^{a,b}\cong (X_n^{a,b})^{op}.
   \]
\end{proposition}
\begin{proof}
    Let $\widehat\varphi\colon M\to M$ be the lattice automorphism given by $\widehat\varphi(0,1)=(0,-1)$
    and $\widehat\varphi(1,0) = (1, -n)$. Since
    $\widehat\varphi_\QQ(\sigma^\vee(n,a,b))=\bar\sigma^\vee{(n,a,b)}$, it induces a variety
    isomorphism $\varphi\colon Y_n^{a,b} \to X_n^{a,b}$.

    Let $x^{a'}y^{b'}$ be a monomial in $\KK[X_n^{a,b}]$. The comultiplication on it is given by
    \eqref{eq:comultipongn}:
    \[
        x^{a'}y^{b'}\mapsto (x\otimes 1 + y^n\otimes x)^{a'}(y\otimes y)^{b'}.
    \]
    The isomorphism $\widehat\varphi$ maps $x^{a'}y^{b'}$ to
    $\varphi^*(x^{a'}y^{b'})=x^{a'}y^{-na'-b'}$. The comultiplication on $\varphi^*(x^{a'}y^{b'})$
    is given by 
    \[
        x^{a'}y^{-na'-b'}\mapsto (x\otimes 1 + y^n\otimes x)^{a'}(y\otimes y)^{-na'-b'}=(x\otimes
        y^{-n}+1\otimes x)^{a'}(y\otimes y)^{-b'},
    \]
    which is opposite to the comultiplication on $x^{a'}y^{b'}$:
    \[
        \varphi^*((x\otimes 1 + y^n\otimes x)^{a'}(y\otimes y)^{b'})=(x\otimes
        1+y^{-n}\otimes x)^{a'}(y\otimes y)^{-b'}.
    \]
\end{proof}
Till the end of the section we will study algebraic properties of $X_n^{a,b}$. Their counterparts
for $Y_n^{a,b}$ can be obtained automatically due to the isomorphism $Y_n^{a,b}\cong
(X_n^{a,b})^{op}$.

\subsection{Central subgroups and quotients}
Recall from Proposition \ref{prop:centreofgn} that the center $Z(G_n)$ of $G_n$ is a cyclic
subgroup of order $n$. 
\begin{proposition}\label{p:monquot}
    Let $C_m\subset Z(G_{mn})$ be a cyclic subgroup of order $m$. The quotient of $X_{mn}^{a,b}$ by
    $C_m$ is isomorphic to $X_n^{a',b'}$, where $a'=\frac{m}{\gcd(m,b)}a$ and
    $b'=\frac{1}{\gcd(m,b)}b$.
\end{proposition}
\begin{proof}
    The homomorphism $\varphi \colon G_{mn}\to G_{mn}/C_m\cong G_n$ is toric and the corresponding
    lattice morphism $\widehat\varphi\colon M\to M$ satisfies $\widehat\varphi(1,0)=(1,0)$,
    $\widehat\varphi(0,1)=(0,m)$. Moreover, $\widehat\varphi(a',b')=(a',mb')=\frac{m}{\gcd(m,b)}(a,b)$, so
    $\widehat\varphi_\QQ(\sigma(n,a',b'))=\sigma(mn, a, b)$. We have a commutative diagram
    % https://q.uiver.app/?q=WzAsNCxbMCwwLCJcXEtLW1hfe21ufV57YSxifV1ee0NfbX0iXSxbMSwwLCJcXEtLW0dfe21ufV1ee0NebX0iXSxbMCwxLCJcXEtLW1hfbl57YScsYid9XSJdLFsxLDEsIlxcS0tbR19uXSJdLFszLDEsIlxcdmFycGhpXioiLDJdLFswLDEsIlxcc3Vic2V0IiwzLHsic3R5bGUiOnsiYm9keSI6eyJuYW1lIjoibm9uZSJ9LCJoZWFkIjp7Im5hbWUiOiJub25lIn19fV0sWzIsMywiXFxzdWJzZXQiLDMseyJzdHlsZSI6eyJib2R5Ijp7Im5hbWUiOiJub25lIn0sImhlYWQiOnsibmFtZSI6Im5vbmUifX19XSxbMiwwLCJcXHZhcnBoaV4qfF97WF9uXnthJyxiJ319IiwyXV0=
    \[\begin{tikzcd}
        {\KK[X_{mn}^{a,b}]^{C_m}} & {\KK[G_{mn}]^{C^m}} \\
        {\KK[X_n^{a',b'}]} & {\KK[G_n]}
        \arrow["{\varphi^*}"', from=2-2, to=1-2]
        \arrow["\subset"{marking}, draw=none, from=1-1, to=1-2]
        \arrow["\subset"{marking}, draw=none, from=2-1, to=2-2]
        \arrow["{\varphi^*|_{X_n^{a',b'}}}"', from=2-1, to=1-1]
    \end{tikzcd}\]
    with $\varphi^*$ being an isomorphism. Since $\widehat\varphi_\QQ$ maps $\sigma(n,a',b')$
    bijectively onto $\sigma(mn, a, b)$, $\varphi^*|_{X_n^{a',b'}}$ is also an isomorphism.
\end{proof}
As a special case, we obtain that every monoid is a quiotient of a monoid on $\mathbb{A}^2$ from Corollary \ref{cor:affmonoids}. Explicitly, we have 
$
X_{an}^{1,b}/C_a \cong X_n^{a,b}.
$

\subsection{Boundary divisor}
For an algebraic monoid $X$ we denote by $R(X)$ the subset of non-invertible elements, so
$R(X)=X\setminus G(X)$. We call it the \emph{boundary divisor}. Obviously, $R(X)$ is closed. It
turns out that in the two-dimensional case $R(X)$ is irreducible as well. An element $\mathbf{0}\in
R(X)$ is called \emph{zero} if for every element $\tx\in X$ we have $\tx\cdot
\mathbf{0}=\mathbf{0}\cdot \tx=\mathbf{0}$.
\begin{theorem} \label{t:noninv}
    The following holds for the boundary divisor $R=R(X_n^{a,b})$ of $X_n^{a,b}$ for integers
    $a,b>0$:
    \begin{enumerate}
        \item $R$ is a closed, irreducible subvariety, which is isomorphic to the affine line
            $\mathbb{A}^1$;            
        \item there is a zero $\mathbf{0}\in R$ and the restriction of the multiplication to $R$ is
            given by
            \[
                \tx\cdot\ty = \mathbf{0}
            \]
            for $\tx, \ty \in R$.

        \item identify $R$ with $\mathbb{A}^1$ such that the zero $\mathbf{0}\in R$ is identified
            with $0\in \mathbb{A}^1$. Then the action of the group $G_n\subset X_n^{a,b}$ by left
            and right multiplications on $R$ is given by
            \[
                (\alpha, \tau)\cdot \tx = \tau^{b+an}\tx,
            \]
            \[
                \tx \cdot (\alpha, \tau)=\tau^b\tx,
            \]
            for $(\alpha, \tau)\in G_n$ and $\tx\in R=\mathbb{A}^1$. 
    \end{enumerate}
\end{theorem}
\begin{proof}
    We first give a proof for monoids on affine spaces as in Corollary \ref{cor:affmonoids}.
    Then we use the isomorphism from Proposition \ref{p:monquot}.

    Recall that the surface $X_n^{a,b}$ is an affine plane if and only if $a=1$. In this case, the
    multiplication is given by 
    \[
        (\tx_1, \ty_1)\cdot (\tx_2, \ty_2) = (\tx_1\ty_2^b+\ty_1^{b+n}\tx_2,
        \ty_1\ty_2),
    \]
    for $(\tx_i, \ty_i)\in \mathbb{A}^2$. The subgroup $G_n$ is given by pairs $(\alpha,
    \tau)\in\mathbb{A}^2$ with nonzero $\tau$. The set of non-invertible elements is 
    \[
        R(X_n^{1,b})=\{(\tx, 0) \colon \tx \in \KK\},
    \]
    so $R(X_n^{1,b})\cong \mathbb{A}^1$. This proves (1), while (2) and (3) are obvious from the
    multiplication formula above. The zero is given by $\mathbf0 = (0,0)\in R$.

    To prove the general case consider the isomorphism $X_n^{a,b}\cong X_{an}^{1,b}/C_a$ from Proposition \ref{p:monquot}. We have 
    \[
        R(X_n^{a,b})=R(X_{an}^{1,b})/C_a\cong \mathbb{A}^1,
    \]
    because the quotient of $\mathbb{A}^1$ by a cyclic group is always isomorphic to $\mathbb{A}^1$
    itself. We also obtain (2) and (3), since the quotient preserves multiplication.
\end{proof}
Theorem \ref{t:noninv} says nothing about $X_n^{1,0}$. This case is special. As a surface
$X_n^{1,0}$ is isomorphic to $\mathbb{A}^2$ with multiplication
\[
    (\tx_1, \ty_1)\cdot (\tx_2, \ty_2) = (\tx_1+\ty_1^{n}\tx_2,
    \ty_1\ty_2).
\]
The subset of non-invertible elements is also given by 
\[
    R(X_n^{1,0})=\{(\tx, 0) \colon \tx \in \KK\},
\]
but the multiplication is
\[
    (\tx_1, 0)\cdot (\tx_2, 0) = (\tx_1, 0).
\]

The group $G_n\times G_n$ acts on $X_n^{a,b}$ by right and left multiplications. For $\tx\in
X_n^{a,b}$ and $(g,h)\in G_n$ we have $(g,h)\cdot \tx = g\tx h^{-1}$.
\begin{corollary}
    Let $X=X_n^{a,b}$ for positive integers $a,b$ with $\gcd(a,b)=1$.  There are three orbits of
    the action of $G_n\times G_n$ on $X$ for positive intgers: the group $G_n\subset X_n^{a,b}$,
    $R(X_n^{a,b})\setminus \{\mathbf{0}\}$, and $\{\mathbf{0}\}$. 
\end{corollary}
\begin{proof}
    It follows immediately from Theorem \ref{t:noninv}(3). 
\end{proof}


\medskip

\begin{thebibliography}{99}

\bibitem{ABZ20}
Ivan Arzhantsev, Sergey Bragin, and Yulia Zaitseva,
\emph{{Commutative algebraic monoid structures on affine spaces}},
Commun. Contemp. Math. \textbf{22} (2020), no. 8, 1950064, 23 pp.
%MR4142327

\bibitem{AK15}
Ivan {Arzhantsev} and Polina {Kotenkova},
\emph{Equivariant embeddings of commutative linear algebraic groups of corank one},
Doc. Math. \textbf{20} (2015), 1039--1053.
%MR3424473

\bibitem{CLS11}
David A. {Cox}, John B. {Little}, and Henry K. {Schenck}.
\emph{Toric varieties},
Graduate Studies in Mathematics, 124. American Mathematical Society, Providence, RI, 2011. xxiv+841 pp.
%MR2810322

\bibitem{DZ21}
Sergey {Dzhunusov} and Yulia {Zaitseva},
\emph{{Commutative algebraic monoid structures on affine surfaces}},
Forum Math. \textbf{33} (2021), no. 1, 177--191. 
%MR4193482

\bibitem{F17}
Gene Freudenburg,
\emph{Algebraic theory of locally nilpotent derivations},
Second edition. Encyclopaedia of Mathematical Sciences, 136. Invariant Theory and Algebraic Transformation Groups, VII. Springer-Verlag, Berlin, 2017. xxii+319 pp.
%MR3700208

\bibitem{H75}
James E. Humphreys,
\emph{Linear Algebraic Groups},
Graduate Texts in Mathematics, No. 21. Springer-Verlag, New York-Heidelberg, 1975. xiv+247 pp.
%MR0396773

\bibitem{N92}
Karl-Hermann Neeb,
\emph{Toric varieties and algebraic monoids},
 Sem. Sophus Lie \textbf{2} (1992), no. 2, 159--187.
 %MR1209135 

\bibitem{P88}
Mohan Putcha,
\emph{Linear Algebraic Monoids},
London Mathematical Society Lecture Note Series, 133. Cambridge University Press, Cambridge, 1988. x+171 pp.
%MR0964690

\bibitem{R05}
Lex Renner,
\emph{Linear Algebraic Monoids},
Encyclopaedia of Mathematical Sciences, 134. Invariant Theory and Algebraic Transformation Groups, V. Springer-Verlag, Berlin, 2005. xii+246 pp.
%MR2134980

\bibitem{R98}
Alvaro {Rittatore},
\emph{Algebraic monoids and group embeddings},
Transform. Groups \textbf{3} (1998), no. 4, 375--396.
%MR1657536

\bibitem{R07}
Alvaro Rittatore, 
\emph{Algebraic monoids with affine unit group are affine},
Transform. Groups \textbf{12} (2007), no. 3, 601--605.
%MR2356324


\bibitem{SR17}
Walter R. Ferrer Santos and Alvaro Rittatore,
\emph{Actions and Invariants of Algebraic Groups},
Second edition. Monographs and Research Notes in Mathematics. CRC Press, Boca Raton, FL, 2017. xx+459 pp.
%MR3617213

\bibitem{V95}
Ernest Vinberg,
\emph{On reductive algebraic semigroups},
 Lie groups and Lie algebras: E. B. Dynkin's Seminar, 145--182,
Amer. Math. Soc. Transl. Ser. 2, 169, Adv. Math. Sci., 26, Amer. Math. Soc., Providence, RI, 1995.




\end{thebibliography}
\end{document}